\documentclass[11pt,reqno]{article}

\usepackage[latin1]{inputenc}

\usepackage{pgf}

\usepackage{amsmath,amsfonts,amssymb,amscd,amsthm,graphicx}

\usepackage{color}
\usepackage{float}
\usepackage{sectsty}
\allsectionsfont{\centering}

\usepackage{fancyhdr}
\usepackage{amsfonts}
\usepackage{latexsym}

\usepackage{amsmath}
\usepackage{amsthm}
\usepackage{makeidx}
\usepackage{latexsym}
\usepackage{graphicx}
\usepackage{amssymb}

\usepackage{hyperref}
\usepackage{float}
\newtheorem{theorem}{Theorem}

\newtheorem{proposition}{Proposition}
\newtheorem{corollary}{Corollary}

\newtheorem{definition}{Definition}
\theoremstyle{remark}

\usepackage[latin1]{inputenc}

\title{\textbf{THE BIHARMONIC MEAN}}

\author{MARCO ABRATE, STEFANO BARBERO, \\UMBERTO CERRUTI, NADIR MURRU}

\date{}

\begin{document}

\maketitle

{\scriptsize{\noindent \emph{University of Turin, Department of Mathematics\\ Via Carlo Alberto 10, Turin, Italy \\ marco.abrate@unito.it, stefano.barbero@unito.it \\ umberto.cerruti@unito.it, nadir.murru@unito.it}}

\begin{abstract}
{\scriptsize 
We briefly describe some well--known means and their properties, focusing on the relationship with integer sequences. In particular, the harmonic numbers, deriving from the harmonic mean, motivate the definition of a new kind of mean that we call the \emph{biharmonic mean}. The biharmonic mean allows to introduce the biharmonic numbers, providing a new characterization for primes. Moreover, we highlight some interesting divisibility properties and we characterize the semi--prime biharmonic numbers showing their relationship with linear recurrent sequences that solve certain Diophantine equations.
}
\end{abstract}

%arithmetic mean \sep contraharmonic mean \sep geometric mean \sep harmonic mean \sep harmonic numbers \sep integer sequences
%\MSC[2010] 11N80 \sep 26E60


\section{INTRODUCTION}
\normalsize
The need to explore the Nature and establish from direct observations
its rules, encouraged the ancient thinkers in finding appropriate mathematical tools, able to extrapolate  numerical data. The \emph{arithmetic mean} is one of the oldest quantities introduced for this purpose, in order to find an unique approximate value of some physical quantity from the empirical data. It has been probably used for the first time in the third century B.C., by the ancient Babylonian astronomers, in their studies on the positions and motions of celestial bodies. The mathematical relevance of the arithmetic mean has been enhanced by the Greek astronomer Hipparchus (190--120 B.C.). Some other Greek mathematicians, following the Pythagoric ideals, have also introduced and rigorously defined further kinds of means. For example, Archytas (428--360 B.C.) named the \emph{harmonic mean} and used it in the theory of music and in the algorithms for doubling the cube. His disciple Eudoxus (408--355 B.C.) introduced the \emph{contraharmonic mean} in his studies on proportions. At the same time as the practical use of many numerical means in various sciences, a deep exploration of their arithmetic and geometric properties took place during the centuries. The book of Bullen \cite{Bul} is a classical reference for a good survey about the various kinds of means and their history.
 
In this paper, we especially focus our attention on some arithmetic aspects related to the most used means in many fields of mathematics.
We also define a new kind of mean, showing how it also allows to give a new characterization for the prime numbers. 

We start recalling the classical definitions and properties of the involved means.

\begin{definition} \label{medie}
Let $a_1,a_2,\ldots, a_t$ be $t$ positive real numbers, then we define their
\begin{itemize}
\item \emph{arithmetic mean} 
\begin{equation}\label{Arith1}{\cal{A}}(a_1,a_2,\ldots, a_t)= \frac{a_1+a_2+\cdots+ a_t}{t}; \quad \end{equation}
\item \emph{geometric mean} 
\begin{equation}\label{Geom1}{\cal{G}}(a_1,a_2,\ldots, a_t)= \sqrt[t]{a_1a_2 \cdots a_t}; \quad \end{equation}
\item  \emph{harmonic mean} 
\begin{equation}\label{Harm1}{\cal{H}}(a_1,a_2,\ldots, a_t)= \left(\frac{1}{t}\left(\frac{1}{a_1}+\frac{1}{a_2}+\cdots+ \frac{1}{a_t}\right)\right)^{-1}; \quad \end{equation}
\item \emph{contraharmonic mean} 
\begin{equation}\label{Contr1}{\cal{C}}(a_1,a_2,\ldots, a_t)= \frac{a_1^2+a_2^2+\cdots+ a_t^2}{a_1+a_2+\cdots+ a_t}. \quad \end{equation}
\end{itemize}
\end{definition}
We have the well--known inequalities 
\[
{\cal{H}}(a_1,a_2,\ldots, a_t)\leq {\cal{G}}(a_1,a_2,\ldots, a_t) \leq {\cal{A}}(a_1,a_2,\ldots, a_t) \leq {\cal{C}}(a_1,a_2,\ldots, a_t).
\]
A very interesting problem is to determine whether at least one of these equalities holds. If we have only two positive real numbers $a$ and $b$, we always obtain the following relations:

\begin{equation} \label{ari}
{\cal{A}}(a,b)={\cal{A}}({\cal{H}}(a,b),{\cal{C}}(a,b)) ,
\end{equation}

\begin{equation} \label{geo}
{\cal{G}}(a,b)={\cal{G}}({\cal{H}}(a,b),{\cal{A}}(a,b)).
\end{equation}

In $1948$ Oystein Ore (\cite{ore1}, \cite{ore2}) introduced the idea of the \emph{harmonic number} finding related properties and a quite surprising answer to this question. He evaluated the four means of all the divisors of a positive integer $n$. Let us denote the set of divisors of $n$ as
$$D(n)=\left\lbrace d_1,d_2,\ldots , d_t\right\rbrace.$$
For the sake of simplicity we pose
\begin{eqnarray*}
  A(n) &=& \mathcal{A}(d_1,d_2,\ldots,d_t),\\
  G(n) &=& \mathcal{G}(d_1,d_2,\ldots,d_t),\\
  H(n) &= &\mathcal{H}(d_1,d_2,\ldots,d_t),\\
  C(n) &=& \mathcal{C}(d_1,d_2,\ldots,d_t).
\end{eqnarray*}
Recalling that the divisor function is
\begin{equation} \label{sigma}\sigma_{x}(n)=\sum_{i=1}^t d_i^x, \quad  x\in \mathbb{N}\end{equation}
the first immediate result of Ore can be summarized in the following theorem.
\begin{theorem}\label{expl}
\begin{equation}\label{Arithore}A(n)= \frac{\sigma_1(n)}{\sigma_0(n)}, \quad \end{equation}
\begin{equation}\label{Geomore} G(n)= \sqrt{n}, \quad \end{equation}
\begin{equation}\label{Harmore} H(n)=\frac{n\sigma_0(n)}{\sigma_1(n)}, \quad \end{equation}
\begin{equation}\label{Controre}C(n)= \frac{\sigma_2(n)}{\sigma_1(n)}. \quad \end{equation}
\end{theorem}

\begin{proof} 
We give a proof only for equalities (\ref{Harmore}) and (\ref{Geomore}) since
(\ref{Arithore}) and (\ref{Controre}) clearly arise from (\ref{Arith1}) and (\ref{Contr1}), respectively.
From equation (\ref{Harm1}) clearly
$$ H(n) = \left(\frac{1}{t}\left(\frac{1}{d_1}+\frac{1}{d_2}+\cdots+ \frac{1}{d_t}\right)\right)^{-1}    $$
and the sum $\frac{1}{d_1}+\frac{1}{d_2}+\cdots+ \frac{1}{d_t}$, when reduced to the least common denominator $n$, gives as numerator $\sigma_1(n)$ (the sum of all divisors of $n$). Moreover, since $t=\sigma_0(n)$ (the number of all divisors of $n$), relation (\ref{Harmore}) easily follows. Now, to prove equation (\ref{Geomore}), we start from (\ref{Geom1})
$$G(n) = \sqrt[t]{d_1 d_2 \cdots d_t}=\sqrt[\sigma_0(n)]{\prod_{i=1}^{\sigma_0(n)} d_i}$$
distinguishing two cases: $n=m^2$ or $n\not =m^2$.
When $n\not=m^2$, $n$ has an even number of divisors, so we multiply  $d_i$ by $\frac{n}{d_i}$ for all $i=1,\ldots,\frac{\sigma_0(n)}{2}$, finding  $\prod_{i=1}^{\sigma_0(n)} d_i=n^{\frac{\sigma_0(n)}{2}}$.
On the other hand, if $n=m^2$, then $t=\sigma_0(n)$ is odd. Similarly, we multiply $d_i$ by $\frac{n}{d_i}$, but in this case we can do this only when $d_i \neq m$, finding
$$ \prod_{i=1}^{\sigma_0(n)} d_i= n^{\frac{\sigma_0(n)-1}{2}}m=(m^2)^{\frac{\sigma_0(n)-1}{2}}m=m^{\sigma_0(n)}.$$
Thus
$$G(n) = \sqrt[\sigma_0(n)]{\prod_{i=1}^{\sigma_0(n)}d_i}=\sqrt[\sigma_0(n)]{m^{\sigma_0(n)}}=m=\sqrt{n}.$$
\end{proof}

A straightforward consequence, observed by Ore in \cite{ore1}, shows that a similar equality to (\ref{geo}) holds taking into account elements of $D(n)$.
\begin{corollary}
For any positive integer $n$ 
\begin{equation} \label{geo2}
G(n)=\mathcal{G}(H(n),A(n)). 
\end{equation}
\end{corollary}
\begin{proof}
By previous theorem, we clearly obtain
$$\mathcal{G}(H(n),A(n))=\sqrt{H(n)\cdot A(n)}=\sqrt{\frac{n \sigma_0(n)}{\sigma_1(n)} \frac{\sigma_1(n)}{\sigma_0(n)}}= \sqrt{n}=G(n).$$
\end{proof}

Equality (\ref{geo2}) can be interpreted  also as a formal identity.
For example when $n=p^2 q$, with $p$ and $q$ primes, $D(n)=\{1,p,q,p q,p^2,p^2q\}$ and, by (\ref{geo2}), we have
\begin{equation}\label{pq}\mathcal{G}(1,p,q,p q,p^2,p^2q)=\mathcal{G}(\mathcal{H}(1,p,q,p q,p^2,p^2q),\mathcal{A}(1,p,q,p q,p^2,p^2q)).\end{equation}
The astonishing fact is that this equality also holds substituting  $p$ and $q$ with any other couple of positive real numbers.
For example, if we use $p=\sqrt{2}$ and $q=\pi$, we obtain the identity
$$ \mathcal{G}(1,\sqrt{2},\pi,\sqrt{2}\,\pi,2,2 \pi)=\mathcal{G}(\mathcal{H}(1,\sqrt{2},\pi,\sqrt{2}\,\pi,2,2 \pi),\mathcal{A}(1,\sqrt{2},\pi,\sqrt{2}\,\pi,2,2 \pi)). $$
which is not so immediate. As previously observed, for randomly chosen positive distinct real numbers $a_1,\ldots,a_t$ the equality
$$\mathcal{G}(a_1,a_2,\ldots,a_t)=\mathcal{G}(\mathcal{H}(a_1,a_2,\ldots,a_t),\mathcal{A}(a_1,a_2,\ldots,a_t)),$$
is false.

The second question is pretty natural: when do the means of the divisors of an integer $n$ also give a result which is an integer?
The case of $G(n)$ is not so interesting because $G(n)$ is an integer if and only if $n$ is a square. The integers $n$ for which $A(n)$ is an integer form the sequence A003601 in OEIS \cite{oeis}:
\begin{equation} \label{aritmetici}
 1, 3, 5, 6, 7, 11, 13, 14, 15, 17, 19, 20, 21, 22, 23, 27, 29, 30, 31, 33, 35, 37, \ldots  
\end{equation}
These integers are the so--called \textit{arithmetic numbers}.

Moreover, every integer $n$ giving an integer value for $C(n)$ belongs to the sequence A020487:
$$1, 4, 9, 16, 20, 25, 36, 49, 50, 64, 81, 100, 117, 121, 144, 169, 180, 196, 200, 225, \ldots.$$

The most interesting case is related to the harmonic mean. Ore provided the following definition.
\begin{definition} \label{armonico} 
A positive integer $n$ is called a \emph{harmonic (divisor) number} (or \emph{Ore number}) if $H(n)$ is an integer.
\end{definition}
The first harmonic numbers are
$$1, 6, 28, 140, 270, 496, 672, 1638, 2970, 6200, 8128, 8190, 18600, 18620, 27846, 30240, \ldots $$
and they form the sequence $A001599.$
The corresponding values of $H(n)$ are listed in $A001600$:
$$ 1, 2, 3, 5, 6, 5, 8, 9, 11, 10, 7, 15, 15, 14, 17, 24, \ldots .$$
Ore also proved that all perfect numbers are harmonic numbers.

In the next section, moving from the above beautiful properties, we will define a new kind of mean which we will call a \emph{biharmonic mean}. We will also define the biharmonic numbers, which provide a new characterization for prime numbers. Moreover, the composite biharmonic numbers will lead to the study of interesting divisibility properties, involving consecutive terms of linear recurrent sequences and solutions of Diophantine equations.

\section{BIHARMONIC MEAN AND BIHARMONIC NUMBERS}

\begin{definition} \label{biarmonica1}
For positive real numbers $a_1,a_2,\ldots,a_t$, we define the \emph{biharmonic mean} as
$$ {\cal{B}}(a_1,a_2,\dots,a_t)={\cal{A}}({\cal{H}}(a_1,\dots,a_t),{\cal{C}}(a_1,\dots,a_t))=\frac{{\cal{H}}(a_1,a_2,\dots,a_t)+{\cal{C}}(a_1,a_2,\dots,a_t)}{2}$$
which corresponds to the arithmetic mean of the harmonic and contraharmonic means of $a_i$'s.
\end{definition}

From Eq. (\ref{ari}), we know that biharmonic mean is equal to the arithmetic mean when $t=2$. But we also know that Eq. \eqref{ari} is not necessarily true when $t>2$. Following Ore's idea, we define the biharmonic numbers as functions similar to $H(n)$.

\begin{definition} \label{biarmonica2}
Let us consider a positive integer $n$ with $D(n)=\left\lbrace d_1,d_2,\dots,d_t\right\rbrace$.
We define $B(n)$ as the biharmonic mean of the divisors of $n$
$$ B(n)={\cal{B}}(d_1,d_2,\dots,d_t).$$
We call an integer $n$ \textit{biharmonic number} if $B(n)$ is an integer.
\end{definition}

From this definition we have
$$ B(n)={\cal{B}}(d_1,d_2,\dots,d_t)=\frac{{\cal{H}}(d_1,d_2,\dots,d_t)+{\cal{C}}(d_1,d_2,\dots,d_t)}{2}=\frac{H(n)+C(n)}{2}, $$
and, by Theorem  \ref{expl}, we can find a closed form for $B(n)$: 
\begin{equation} \label{Hn} 
B(n)=\frac{H(n)+C(n)}{2} = \frac{\frac{n \sigma_0(n)}{\sigma_1(n)}+\frac{\sigma_2(n)}{\sigma_1(n)}}{2}=\frac{n \sigma_0(n)+\sigma_2(n)}{2\sigma_1(n)}.
\end{equation}

Investigating the occurrence of biharmonic numbers among  positive integers, we find the sequence
$$1, 3, 5, 7, 11, 13, 17, 19, 23, 29, 31, 35, 37, 41, 43, 47, 53, 59, 61, \ldots $$
which is very similar to the sequence of prime numbers. This is not so strange, because if $n$ is prime $B(n)=A(n)$.
\begin{theorem} \label{primi} Every odd prime number $p$ is a biharmonic number and $B(p)=\cfrac{p+1}{2}$ .
\end{theorem}
\begin{proof}
If $p$ is an odd prime we have $\sigma_0(p)=2, \sigma_1(p)=1+p, \sigma_2(p)=1+p^2$, thus
$$ B(p)=\frac{2p+1+p^2}{2(1+p)}=\frac{p+1}{2} .$$
\end{proof}

The very interesting fact is that we can \emph{characterize} odd prime numbers using $B(n)$. Indeed, we can prove that the converse of the previous theorem is also true.
\begin{theorem} \label{primi2}
If for odd integer $n\not =1$, $B(n)=\frac{n+1}{2}$, then $n$ is a prime.
\end{theorem}
\begin{proof}
First, we observe that the equality $B(n)=\frac{n+1}{2}$ corresponds to \begin{equation}\label{B}(n+1)\sigma_1(n)-(\sigma_2(n)+n\sigma_0(n))=0.\end{equation}
Let us consider two cases: $n=k^2$ or $n\not=k^2$.
When  $n=k^2$ with $k\not=1$, we have $\sigma_0(n)=2m+1$ for some $m>0$ and  
$$D(n)=\left\lbrace d_1, d_2,...,d_m,d_{m+1},d_{m+2},...,d_{2m},k\right\rbrace,$$ where we pose $d_i=\frac{n}{d_{m+i}}$ for $i=1,...,m$.\\
Clearly, by definition  
\begin{equation}\label{sigma1-t}\sigma_1(n)=\sum_{i=1}^{2m}d_i+k\end{equation}
and
\begin{equation}\label{sigma2-t}\sigma_2(n)=\sum_{i=1}^{2m}d_{i}^2+k^2.\end{equation}
From (\ref{sigma1-t}), we find that
$$(n+1)\sigma_1(n)=\sum_{i=1}^{2m}nd_i+nk+\sum_{i=1}^{2m}d_i+k=\sum_{i=1}^m d_{i}^2d_{m+i}+\sum_{i=1}^m d_id_{m+i}^2+nk+\sum_{i=1}^{2m}d_i+k .$$
Rearranging the terms and remembering that $n=k^2$, we obtain
\begin{equation}\label{pt}
(n+1)\sigma_1(n)=\sum_{i=1}^m(d_i+d_{m+i})(d_id_{m+i}+1)+k^3+k.\end{equation}
On the other hand, by (\ref{sigma2-t}), the following relation holds
\begin{equation}\label{st}
\begin{split}
\sigma_2(n)+n\sigma_0(n)&=\sum_{i=1}^{2m}d_i^2+k^2+(2m+1)n=
\sum_{i=1}^{2m}d_i^2+\sum_{i=1}^{2m}d_i\frac{n}{d_i}+2k^2=\\
&=\sum_{i=1}^{2m}d_i^2+2\sum_{i=1}^m d_id_{m+i}+2k^2=\sum_{i=1}^m(d_i+d_{m+i})^2+2k^2.
\end{split}
\end{equation}
Now, using (\ref{pt}) and (\ref{st}), we obtain
$$(n+1)\sigma_1(n)-(\sigma_2(n)+n\sigma_0(n))=\sum_{i=1}^m(d_i+d_{m+i})(d_id_{m+i}+1)+k^3+k-2k^2-\sum_{i=1}^m(d_i+d_{m+i})^2,$$
and we finally get
\begin{equation}\label{terza-t}
(n+1)\sigma_1(n)-(\sigma_2(n)+n\sigma_0(n))=\sum_{i=1}^m(d_i+d_{m+i})(d_i-1)(d_{m+i}-1)+k(k-1)^2.
\end{equation}
This equality tells us that the first member of (\ref{B}) will never be equal to $0$ if $n\not=1$.\\
Now, let us consider $n\not=k^2$, $\sigma_0(n)=2m$ for some $m\geq1$. With similar calculations, we find
\begin{equation}\label{terza-t2}
(n+1)\sigma_1(n)-(\sigma_2(n)+n\sigma_0(n))=\sum_{i=1}^m(d_i+d_{m+i})(d_i-1)(d_{m+i}-1).
\end{equation}
When $m>1$, the only summand equal to $0$ corresponds to the couple of trivial divisors $1$ and $n$. Thus, equality (\ref{B}) occurs only when $\sigma_0(n)=2$, or, in other words, if $n$ is prime.
\end{proof}
Clearly, it is interesting to study properties of composite biharmonic numbers. The non--prime biharmonic numbers which are similar to the prime numbers the most are the semiprime biharmonic numbers, i.e., numbers $n$ such that $n=pq$, for $p,q$ primes, and $B(n)\in \mathbb N$. In this case, we have
$$B(n)=B(pq)=\cfrac{(p+q)^2+(pq+1)^2}{2(p+1)(q+1)}.$$
Semiprime biharmonic numbers belong to a wider set of integers that we will call \emph{crystals} for their beautiful properties. Let us consider the following function defined on integers:
\begin{equation} \label{bcrystal} \mathbf{B}(a,b)=\cfrac{(a+b)^2+(ab+1)^2}{2(a+1)(b+1)}. \end{equation}
\begin{definition}
An odd number $n$ is called a \emph{crystal} if $n=ab$, with $a,b>1$ and $\mathbf{B}(a,b)\in\mathbb N$.
\end{definition}
In the following section, we determine all the  crystals by means of a particular linear recurrent sequence.

\section{DIVISIBILITY PROPERTIES}
In this section, we characterize all pairs of odd integers $a,b$ such that $\mathbf{B}(a,b)\in\mathbb N$ by using recurrent sequences and integer points on certain conics.\\
First of all, we highlight that $\mathbf{B}(a,b)\in\mathbb N$ is equivalent to different divisibility properties involving the numbers $a,b$.
\begin{proposition} \label{equiv}
Given two integer odd numbers $a,b$, the following statements are equivalent.
\begin{enumerate}
\item $\mathbf{B}(a,b)\in\mathbb N$
\item $\mathbf{F}(a,b)=\cfrac{(ab+1)^2}{(a+1)(b+1)}\in\mathbb N$
\item $\mathbf{P}(a,b)=\cfrac{(a+b)(ab+1)}{(a+1)(b+1)}\in\mathbb N$
\item $\mathbf{Q}(a,b)=\cfrac{(a+b)^2}{(a+1)(b+1)}\in\mathbb N$
\end{enumerate}
\end{proposition}
\begin{proof} 
Since  
$$\mathbf{B}(a,b)+\mathbf{P}(a,b)=\cfrac{(a+1)(b+1)}{2}, \quad \mathbf{P}(a,b)+\mathbf{Q}(a,b)=a+b, \quad \mathbf{F}(a,b)+\mathbf{P}(a,b)=ab+1,$$ we clearly have 
 $$\mathbf{B}(a,b)\in \mathbb N \Leftrightarrow \mathbf{P}(a,b) \in \mathbb N, \quad
\mathbf{P}(a,b)\in \mathbb N \Leftrightarrow \mathbf{Q}(a,b) \in \mathbb N,
\quad \mathbf{F}(a,b)\in \mathbb N \Leftrightarrow \mathbf{P}(a,b) \in \mathbb N.$$
\end{proof}

In order to characterize crystals, we need some preliminary results about the Diophantine equation 
$$(x+y-1)^2-wxy=0$$
with $x,y$ unknown and $w\in \mathbb{N}$ a given parameter.
This equation has been solved over the positive integers in \cite{abcm} by using a particular recurrent sequence. 
\begin{theorem} \label{u}
The pair of positive integers $(x,y)$ is a solution of the Diophantine equation
$$(x+y-1)^2-wxy=0,$$
with $w\in\mathbb N$, if and only if $(x,y)=(u_{n}(w),u_{n-1}(w))$ for a given index $n\geq1$, where $(u_n(w))_{n=0}^{+\infty}$ is the sequence defined by
$$\begin{cases}  u_0(w)=0,\quad u_1(w)=1 \cr u_{n+1}(w)=(w-2)u_n(w)-u_{n-1}(w)+2,\quad \forall n\geq1 \quad . \end{cases}$$
\end{theorem}
When there will be no possibility of confusion, we will omit the dependence on $w$ from $u_n(w)$.

The sequence $(u_n)_{n=0}^{+\infty}$ can be written as a linear recurrent sequence of order 3:
\begin{equation} \label{u3} \begin{cases} u_0=0,\quad u_1=1,\quad u_2=w, \cr u_{n+2}=(w-1)u_{n+1}-(w-1)u_{n}+u_{n-1},\quad \forall n\geq1 \quad . \end{cases}\end{equation}
Indeed, if $(p_n)_{n=0}^{+\infty}$ is a linear recurrent sequence of order $m$ with characteristic polynomial $f(t)=t^m-\sum_{h=1}^{m}f_ht^{m-h}$ and initial conditions $p_0,...,p_{m-1}$, then the sequence $(q_n)_{n=0}^{+\infty}$ satisfying the recurrence
$$q_m=\sum_{h=1}^m{f_h q_{m-h}}+k$$
and initial conditions $p_0,...,p_{m-1}$ is a linear recurrent sequence of degree $m+1$ with characteristic polynomial $(x-1)f(x)$ and initial conditions $p_0,...,p_{m-1},p_m+k$ (see, e.g., \cite{Cerruti}).
Thus, from
$$x^3-(w-1)x^2+(w-1)x-1=(x-1)(x^2-(w-2)x+1)$$
we have that sequence $(u_n)_{n=0}^{+\infty}$ satisfies the recurrence \eqref{u3}. This sequence is related to the linear recurrent sequence $(a_n(w))_{n=0}^{+\infty}$ defined by
\begin{equation}\label{a} \begin{cases} a_0(w)=0,\quad a_1(w)=1, \cr a_n(w)=\sqrt{w}a_{n-1}(w)-a_{n-2}(w),\quad \forall n\geq2 \quad . \end{cases} \end{equation} 
The relation between sequences $(u_n)_{n=0}^{+\infty}$ and $(a_n)_{n=0}^{+\infty}$ is determined by the Chebyshev's second polynomial
$$\theta(x)=2x^2-1 \quad \forall x\in\mathbb R,$$
and stated in the following proposition.
\begin{proposition}\label{theta}
For any index $n$, we have
$$\theta(a_n)=2u_n-1 .$$
\end{proposition}
\begin{proof}
The sequence $(a_n)_{n=0}^{+\infty}$ recurs with characteristic polynomial $x^2-\sqrt{w}x+1$ whose companion matrix is
$$F=\begin{pmatrix} 0 & 1 \cr -1 & \sqrt{w} \end{pmatrix}.$$
By definition of $\theta$, we have $\theta(a_n)=2a_n^2-1$, $\forall n\geq0$ and the sequence $(\theta(a_n))_{n=0}^{+\infty}$ is a linear recurrent sequence whose characteristic polynomial is the characteristic polynomial of the matrix $F\otimes F$, where $\otimes$ is the Kronecker product (see \cite{Cerruti}). In this case, we have
$$F\otimes F=\begin{pmatrix} 0 & 0 & 0 & 1 \cr 0 & 0 & -1 & \sqrt{w} \cr 0 & -1 & 0 & \sqrt{w} \cr 1 & -\sqrt{w} & -\sqrt{w} & w  \end{pmatrix}$$
whose characteristic polynomial is
$$(x-1)^2(x^2-(w-2)x+1).$$
Thus the minimal polynomial of $(\theta(a_n))_{n=0}^{+\infty}$ is the same  polynomial as the characteristic polynomial of the sequence $(u_n)_{n=0}^{+\infty}$. Finally, observing that
$$\theta(a_0)=2a_0^2-1=2u_0-1=0,\quad \theta(a_1)=2a_1^2-1=2u_1-1=1,$$
$$\theta(a_2)=2a_2^2-1=2u_2-1=2w-1,$$
we have the thesis.
\end{proof}

Two consecutive elements of $(a_n)_{n=0}^{+\infty}$ correspond to a point belonging to the conic
$$C(w)=\{(x,y)\in\mathbb R: x^2+y^2-\sqrt{w}xy=1\}$$
with $w\in\mathbb N$.
\begin{proposition} \label{aC}
For any integer $n>0$, we have
$$(a_n,a_{n-1})\in C(w)$$
\end{proposition}
\begin{proof}
In the proof of the previous proposition, we  observed that $F$ is the companion matrix of the characteristic polynomial of  $(a_n)_{n=0}^{+\infty}$. Thus, we have
$$F^n=\begin{pmatrix} -a_{n-1} & a_n \cr -a_n & a_{n+1}  \end{pmatrix}.$$
Since $\det(F) = 1$, we have $\det (F^n)=1$, i.e.,
$$a_n^2-a_{n-1}a_{n+1}=1$$ and by Equation \eqref{a} 
$$a_n^2-a_{n-1}(\sqrt{w}a_n-a_{n-1})=a_n^2+a_{n-1}^2-\sqrt{w}a_na_{n-1}=1.$$
\end{proof}

In the following proposition, we highlight the relation between points on the conic $C(w)$ and points on the conic
$$C_2(w)=\{(x,y)\in\mathbb R: (x+y-1)^2=wxy\}.$$
\begin{proposition} \label{conics}
Let $C(w), C_2(w), C_3(w)$ be the following conics
$$C(w)=\{(x,y)\in\mathbb R: x^2+y^2-\sqrt{w}xy=1\}, \quad C_2(w)=\{(x,y)\in\mathbb R: (x+y-1)^2=wxy\},$$
$$\quad C_3(w)=\{(x,y)\in\mathbb R: (x+y)^2=w(x+1)(y+1)\},$$
with $w\in\mathbb N$. For any $x,y\in \mathbb R^+$, we have
$$(x,y)\in C(w) \Leftrightarrow (\theta(x),\theta(y))\in C_3(w)$$
$$(x,y)\in C(w) \Leftrightarrow (x^2,y^2)\in C_2(w)$$
$$(x,y)\in C_2(w) \Leftrightarrow (2x-1,2y-1)\in C_3(w).$$
\end{proposition}
\begin{proof}
Remembering that
$$\mathbf{Q}(a,b)=\cfrac{(a+b)^2}{(a+1)(b+1)},$$
we obtain
$$\mathbf{Q}(\theta(x),\theta(y))=\cfrac{(x^2+y^2-1)^2}{x^2y^2}$$
and
$$\mathbf{Q}(\theta(x),\theta(y))=w\Leftrightarrow\cfrac{(x^2+y^2-1)^2}{x^2y^2}=w.$$
Now, we get
$$\cfrac{(x^2+y^2-1)^2}{x^2y^2}=w\Leftrightarrow \cfrac{x^2+y^2-1}{xy}=\sqrt{w}$$
and finally
$$\mathbf{Q}(\theta(x),\theta(y))=w\Leftrightarrow x^2+y^2-\sqrt{w}xy=1.$$
Moreover, we have
$$(x,y)\in C(w)\Leftrightarrow x^2+y^2-1=\sqrt{w}xy$$
and squaring both members
$$(x,y)\in C(w)\Leftrightarrow (x^2+y^2-1)^2=wx^2y^2\Leftrightarrow (x^2,y^2)\in C_2(w).$$
Finally, $(x,y)\in C_2(w) \Leftrightarrow (2x-1,2y-1)\in C_3(w)$ since
$$\mathbf{Q}(2x-1,2y-1)=\cfrac{(x+y-1)^2}{xy}.$$
\end{proof}

Now, we are ready to classify all crystals in the following theorem.
\begin{theorem}
An odd number $N=ab$, with $a,b>1$, is a crystal if and only if there exist two positive integers $w,n$ such that $n\geq3$ and
$$a=\theta(a_n),\quad b=\theta(a_{n-1}).$$ 
\end{theorem}
\begin{proof}
\begin{enumerate}
\item ''$\Leftarrow$''\\
If we have $a=\theta(a_n)$ and $b=\theta(a_{n-1})$ for a positive integer  $n\geq 3$ (i.e., $\theta(a_n)$ and $\theta(a_{n-1})$ are odd positive integers greater than $1$), then by Proposition \ref{conics} $(a,b)\in C_2(w)$ and by previous proposition we have $\mathbf Q(a,b)=w$, so $N=ab$ is a crystal.
\item ''$\Rightarrow$''\\
Let $N=ab$ be a crystal. By Propostion \ref{equiv} there exists a positive integer $w$ such that $\mathbf Q(a,b)=w$ and by Proposition \ref{conics} we know that
$$\left(\cfrac{a+1}{2},\cfrac{b+1}{2}\right)\in C_2(w).$$
Thus, by Theorem \ref{u} there exists an index $n\geq3$ such that
$$\left( \cfrac{a+1}{2}, \cfrac{b+1}{2} \right)=(u_n,u_{n-1})$$
and finally by Proposition \ref{theta} we obtain
$$a=\theta(a_n),\quad b=\theta(a_{n-1}).$$
\end{enumerate}
\end{proof}

\section{CONCLUSION}

The biharmonic numbers show many interesting aspects, principally related to their divisibility properties and their connections to linear recurrent sequences and Diophantine equations. Moreover, the sequence of the biharmonic number has been already included in OEIS (sequence A$210494$), starting from the definition provided by Umberto Cerruti. The biharmonic numbers and the crystals appear to deserve a deepening study. For example, the authors conjecture that if $n=ab$ is a crystal, then there not exists another couple of positive integers $c,d>1$, different from the couple $a,b$, such that $n=cd$ and $B(c,d)\in\mathbb N$, i.e., the components of the crystals are unique. However, it seems that the proof (or a counterexample) does not follow easily.

\end{document}